\documentclass[a4paper]{amsart}

\usepackage{hyperref}
\usepackage[english]{babel}
\usepackage[utf8]{inputenc}
\usepackage[babel]{csquotes}
\usepackage[backend=bibtex,style=alphabetic,doi=false]{biblatex}
\usepackage{amsmath}
\usepackage{amssymb}
\usepackage{amsthm} 
\usepackage{mathtools}
\usepackage{mathrsfs, stmaryrd}
\usepackage{enumitem}
\usepackage[margin=1.2in]{geometry}
\usepackage{textcomp}

\title[Symplectic resolutions for Higgs bundles]{Symplectic resolutions for Higgs moduli spaces}

\author{Andrea Tirelli}
\address{Department of Mathematics, Imperial College, London, 180 Queen’s Gate, London SW7 2AZ, UK}
\email{a.tirelli15@imperial.ac.uk}
\date{}

\newcommand{\mc}[1]{\mathcal{#1}}
\newcommand{\mb}[1]{\mathbb{#1}}

\newcommand{\mf}[1]{\mathfrak{#1}}
\newcommand{\mr}[1]{\mathrm{#1}}
\newcommand{\tit}[1]{\textit{#1}}

\newcommand{\mm}{\mc{M}}

\newtheorem{thm}{Theorem}[section]
\newtheorem*{mainthm*}{Main result}
\newtheorem{lem}[thm]{Lemma}

\newtheorem{prop}[thm]{Proposition}

\theoremstyle{definition}
\newtheorem{defn}[thm]{Definition}

\theoremstyle{remark}
\newtheorem*{claim*}{Claim}
\newtheorem{rem}[thm]{Remark} 
\bibliography{bibliography.bib}

\begin{document}
\begin{abstract}
In this paper, we study the algebraic symplectic geometry of the singular moduli spaces of Higgs bundles of degree $0$ and rank $n$ on a compact Riemann surface $X$ of genus $g$. In particular, we prove that such moduli spaces are symplectic singularities, in the sense of Beauville \cite{beauville}, and admit a projective symplectic resolution if and only if $g=1$ or $(g, n)=(2,2)$. These results are an application of a recent paper by Bellamy and Schedler \cite{bellamy-schedler} via the so-called Isosingularity Theorem.
\end{abstract}
\maketitle
\section{Introduction}\label{intro}
In this paper we show how a recent result of Bellamy and Schedler \cite{bellamy-schedler} on symplectic resolutions of quiver and character varieties can be used to derive information on symplectic resolutions for the moduli space $\mc{M}_H(X, n)$ of semistable Higgs bundles of degree 0 and rank $n$ on a compact Riemann surface $X$ of genus $g$. In particular, we prove that for $g > 1$ and $(g, n)\neq (2,2)$ the aforementioned moduli space  does not admit such a resolution. On the other hand, we show that, in the case of elliptic curves, $\mm_H(X, n)$ does admit a symplectic resolution (note that in the case $(g, n)=(2,2)$ such a resolution was constructed in \cite{kiem-yoo}). For the proof of these results, a central tool is the so-called Isosingularity Theorem, proved by Simpson in the seminal paper \cite{simpson1994-2}.
\par Higgs bundles and symplectic resolutions have become ubiquitous throughout algebraic and differential geometry, representation theory and mathematical physics. For instance, Higgs bundles, which first emerged thirty years ago in Nigel Hitchin’s study of the self-duality
equations on a Riemann surface \cite{hitchin1987} and in Carlos Simpson’s work on nonabelian Hodge theory, \cite{simpson1994-1, simpson1994-2}, play a role in many different areas of mathematics, including gauge theory, Kähler and hyperkähler geometry, surface group representations, integrable systems, nonabelian Hodge theory, the
Deligne–Simpson problem on products of matrices, and (most recently) mirror symmetry and Langlands duality. On the other hand, the theory of (conical) symplectic resolutions has been widely studied not only in mathematics, but also in physics, and has applications and connections to representation theory, symplectic geometry, quantum cohomology, mirror symmetry, and equivariant cohomology. For a survey on some of these connections, see, \tit{e.g.}, \cite{bpw12} and \cite{blpw14}.
\par Symplectic resolutions for Higgs bundles have been considered by Kiem and Yoo in \cite{kiem-yoo}: in their work they prove that $\mm_H(X, 2)$ admits a symplectic desingularization if and only if $g=2$. 

\begin{rem}As a special case of our main theorem, we obtain a new short proof of the aforementioned result form \cite{kiem-yoo} that there is no symplectic resolution of $\mm_H(X, n)$ when $n=2$ and $g\geq 3$.
\end{rem}

The paper is organized as follows: in Section \ref{cvhbsr} we recall some well known results on Higgs bundles and their moduli spaces and give the relevant definitions of symplectic singularities and symplectic resolutions; then, in Section \ref{isosec}, we recall Simpson's Isosingularity Theorem, which is the central result needed for our proof; in Section \ref{cvhbproof}, we state the theorem of \cite{bellamy-schedler} on the non-existence of symplectic resolutions for a certain class of character varieties and explain how one can formulate and prove an analogous statement for the moduli space of Higgs bundles. Moreover, in addition to the $(g, n)=(2,2)$ case, for which we already know from \cite{kiem-yoo} the existence of a symplectic resolution, we prove that, in the elliptic curve case, one has a symplectic resolution.
\begin{mainthm*}[Theorems \ref{hbsing}, \ref{nosympfin} and \ref{symp-ell} below]The following holds true:
 \begin{itemize}
\item[(A)]The moduli spaces $\mm_H(X, n)$ are symplectic singularities. 
\item[(B)] They admit projective symplectic resolutions exactly in the cases $g=1$ and $(g, n)=(2,2)$.
\end{itemize}
\end{mainthm*}
Part (A) of the above result is proved by using Namikawa's criterion \cite[Theorem 6]{namikawa}, Simpson's Isosingularity Theorem (Theorem \ref{isosing}) and the hyperk\"ahler structure on the moduli space of stable Higgs bundles. Furthermore, part (B) follows from a combination of aforementioned Isosingularity Theorem and (a formal analogue of)  one of the main results in \cite{bellamy-schedler}, (Theorem \ref{nosymp}). Finally, to consider the elliptic curve case, a result of Franco \cite{franco} gives a clear geometric description of the moduli space $\mm_H(X, n)$, when $X$ is an elliptic curve, which allows us to see that $\mm_H(X, n)$ does admit a symplectic resolution (Theorem \ref{symp-ell}).

\subsection*{Acknowledgements}The author wishes to thank his PhD supervisor Dr Travis Schedler for the support and guidance given in studying the subject of the present article, Laura Schaposnik for useful comments on a first draft of this paper and Emilio Franco, Indranil Biswas, Marina Logares, Tamas Hausel, Ben Davison and Richard Wentworth for useful conversations on this project. This work was supported by the Engineering and Physical Sciences Research Council [EP/L015234/1], 
The EPSRC Centre for Doctoral Training in Geometry and Number Theory (The London School of Geometry and Number Theory), Imperial College London and University College London.
\section{Character varieties, Higgs bundles and symplectic resolutions}\label{cvhbsr}

In this section we will give the main definitions and results to fix the set up of the paper. For the sake of brevity, we will not give the proofs of any of the statements mentioned, for which we shall refer the reader to the relevant references.
\subsection{Character varieties} Let $X$ be a smooth complex projective curve of genus $g$, and let $G$ be a reductive algebraic group over $\mb{C}$. Consider the space $Y_G=\mr{Hom}(\pi_1(X), G)$ of homomorphisms from the fundamental group of $X$ to the group $G$. Using the presentation by generators and relations of $\pi_1(X)$, we can give $Y_G$ the structure of affine variety: indeed, we know that 
\[
\pi_1(X)\cong \langle a_1, \dots, a_g, b_1,\dots, b_g\rangle/R,
\]
where $R$ is the relation $R=\prod_{i=1}^g[a_i, b_i]$ and $[a, b]$ denotes the commutator $[a, b]:=aba^{-1}b^{-1}$. Then, we can embed $Y_G$ into $G^{2g}$ as follows: 
\[
Y_G\rightarrow G^{2g}, \ \ \rho\mapsto (\rho(a_1), \dots, \rho(b_g)),
\]
which is equivalent to considering $Y_G$ as the subvariety of $G^{2g}$ cut out by the equation \[\prod_{i=1}^g[A_i, B_i]=1,\] for $A_i, B_i \in G$, $i=1, \dots, g$.
Since the group $G$ acts by conjugation on $G$ and one may define the $G$-character variety of $X$ as the categorical quotient \cite{mumford}
\[
\mf{X}(g, G)=Y_G\sslash G.
\]
Algebraically, this is just 
\[
\mf{X}(g, G)=\mr{Spec}(\mb{C}[Y_G]^G),
\]
the spectrum of the ring of $G$-invariant functions on $Y_G$.
\begin{rem}In the notation $\mf{X}(g, G)$ we omitted $X$ since the character variety depends only on the topology of $X$, \tit{i.e.} on its genus, and not on the complex structure.
\end{rem}

Despite their simple definition, character varieties have a very rich geometry and have been the subject of a large body of literature. For the purposes of this paper, we will be interested in the case where $G=GL(n, \mb{C})$ and we will use the notation $\mf{X}(g, n)$ for the character variety $\mf{X}(g, GL(n, \mb{C}))$, where $X$ is a compact Riemann surface of genus $g$. 
\subsection{Higgs bundles} In what follows we shall give a brief overview of Higgs bundles and their moduli spaces. Our main references are \cite{raboso-rayan} and the seminal papers of Hitchin \cite{hitchin1987} and Simpson \cite{simpson1992, simpson1994-1, simpson1994-2}.
\begin{defn} A \tit{Higgs bundle} on $X$ is a pair $(E, \Phi)$, where $E$ is a holomorphic vector bundle on $X$ and $\Phi$, the \tit{Higgs field}, is an $End(E)$-valued 1-form on $X$, \tit{i.e.} $\Phi \in H^0(End(E)\otimes\Omega_X^1)$.
\end{defn}\label{higgsdefn}
In order to have a moduli space with a meaningful geometric structure, one has to consider bundles of fixed rank and degree which satisfy the following stability condition. 
\begin{defn} A Higgs bundle $(E, \Phi)$ on $X$ is \tit{semistable} if for any subbundle $F$ of $E$ such that $\Phi(F)\subset F\otimes\Omega_X^1$, one has 
\[
\mu(F)\leq \mu(E),
\]
where $\mu(E):=\mr{deg}(E)/\mr{rank}(E)$ is the slope of a vector bundle. The Higgs bundle $(E, \phi)$ is said to be \tit{stable} if the strict inequality holds. Moreover, $(E, \phi)$ is said to be \tit{polystable} if it is either stable or a direct sum of stable Higgs bundles with the same slope. 
\end{defn}
\begin{rem}Explaining the origin of the notion of Higgs bundles and the use of such terminology, first introduced by Hitchin, is beyond of the scope of this paper, and thus we would like only to highlight that Higgs bundles were defined in the context of the study of certain self-duality equations on a Riemann surface. The interested reader may wish to refer to the original paper \cite{hitchin1987} and to the references mentioned in the introduction for further details.
\end{rem}
\par As mentioned before, imposing the (semi)stability condition makes it possible to have a moduli space that has sufficiently nice properties. The construction of such a moduli space was firstly carried out by Hitchin, in the rank $2$ case \cite{hitchin1987}, and generalized to arbitrary rank by Nitsure \cite{nitsure}, via the use of Geometric Invariant Theory. We summarise this in the following theorem, which describes the structure of the moduli space in the case when the $g(X)\geq 2$.
\begin{thm}\cite{hitchin1987, nitsure}  Let $g\geq 2$ be an integer and let $\mc{M}_H(n, d)$ be the set of $S$-equivalence classes of semistable Higgs bundles of rank $n$ and degree $d$ on a smooth projective curve $X$ of genus $g$. Then, $\mc{M}_H(X, n, d)$ is a quasi-projective variety, which contains the moduli space of stable Higgs bundles $\mc{M}^s_H(X, n, d)$ as an open smooth subvariety.
\end{thm}
\begin{rem}\label{poly} For the sake of conciseness we omit the definition of $S$-equivalence between Higgs bundles, for which we refer the reader to \cite[Appendix, \S 3.1]{wells}. What will be important for us is that, in an $S$-equivalence class of semistable Higgs bundles, there is only one (up to isomorphism) polystable Higgs bundle; see \cite[Proposition 6.6, Corollary 6.7]{simpson1994-2} for a proof in the degree 0 case. Thus, at least for $d=0$, we may think of $\mm_H(X, n, d)$ as the moduli space of isomorphism classes of polystable Higgs bundles of rank $n$ and degree $d$. 
\end{rem}
In what follows we will use the notation $\mm_H(X, n)$ for the moduli space $\mm_H(X, n, 0)$. $\mc{M}_H(X, n)$ is precisely the object of study of this paper, and in Section \ref{cvhbproof} we will characterize the singularities of such a quasi-projective variety and determine when it admits a symplectic resolution. For this, we shall first recall some basic facts about symplectic resolutions.
\subsection{Symplectic resolutions} The theory of symplectic singularities and symplectic resolutions was first defined by Beauville in \cite{beauville} and since then it has seen a tremendous development, see \tit{e.g.} \cite{fu}.
\begin{defn}Let $Y$ be a normal algebraic variety over $\mb{C}$. Then, we say that $Y$ is a \tit{symplectic singularity} if the smooth locus of $X$, $U=Y\setminus Y^{sing}$, carries a holomorphic symplectic $2$-form $\omega_{U}$ such that, for every resolution of singularities $\rho:\tilde{Y}\rightarrow Y$ the pull-back $\rho^*\omega_{Y_U}$ extends to a holomorphic $2$-form on $\tilde{Y}$.
\end{defn}
\begin{rem}One could alternatively define symplectic singularities by requiring the existence of a resolution of singularities that satisfies the condition of the above definition. It turns out that this is equivalent to requiring the pull-back of the symplectic form to extend for $every$ resolution of singularities. 
\end{rem}
\begin{defn}Given a symplectic singularity $(X, \omega_U)$, we say that a resolution $\rho:\tilde{X}\rightarrow X$ is \tit{symplectic} if the extension of $\rho^*\omega_U$ is a holomorphic symplectic $2$-form. 
\end{defn}
\begin{rem} The reader should refer to \cite[\S 2]{fu}, for a list of examples of symplectic singularities, symplectic resolutions and an account on symplectic singularities which do not admit symplectic resolutions. In this context, Section \ref{cvhbproof} provides a further example of a symplectic singularity which does not admit a symplectic resolution. For related examples, see, \tit{e.g.}, \cite{belsingular, bs1, bs2} on quotient singularities.
\end{rem}
\begin{rem}Note that, following Beauville \cite{beauville}, we define symplectic singularities using holomorphic 2-forms. One can also define them requiring the symplectic form to be algebraic, and it appears that the symplectic structures defined on the moduli spaces under consideration are indeed algebraic.
\end{rem}
\section{The Isosingularity Theorem}\label{isosec}
In this section, we will state one of the two crucial results needed in the proof of our main result, known as the Isosingularity Theorem, and proved by Simpson in his seminal paper \cite{simpson1994-2}, where the general version of the nonabelian Hodge correspondence is given. As suggested by the name, the nonabelian Hodge correspondence can be thought of as a nonabelian version of the well known Hodge Theorem, which gives an isomorphism between the de Rham cohomology $H^n_{dR}(X, \mb{C})$ and the Dolbeaut cohomology $\oplus_{p+q=n}H^q(X,\Omega^p)$ of a compact K\"ahler manifold $X$. Putting this result together with the classical de Rham Theorem, one obtains isomorphisms:
\[
H_B^n(X, \mb{C})\cong H^n_{dR}(X, \mb{C})\cong \bigoplus_{p+q=n}H^q(X, \Omega^p).
\]
In the nonabelian Hodge correspondence, $\mb{C}$ is substituted by a complex algebraic group $G$ and the above cohomology spaces are replaced by the so-called Betti, de Rham and Higgs moduli spaces respectively, which all have a much richer geometric structure than their abelian counterparts. More explicitly, in the case of $G=GL(n, \mb{C})$, one considers the following spaces. 
\begin{itemize}
\item  The Betti moduli space $\mc{M}_B(X, n)$: the space of representations $\pi_1(X)\rightarrow GL(n, \mb{C})$ modulo the conjugation action of $G$, which is also known as the character variety $\mf{X}(g, n)$, recalled in Section \ref{cvhbsr};
\item The de Rham moduli space $\mc{M}_{dR}(X, n)$: the moduli space of flat rank $n$ vector bundles on $X$;
\item The Higgs moduli space $\mc{M}_{H}(X, n)$: the moduli space of semistable Higgs bundles of degree $0$ and rank $n$, which we recalled in Section \ref{cvhbsr}.
\end{itemize}

From the work of Hitchin \cite{hitchin1987}, Donaldson \cite{dona}, Corlette \cite{corlette} and Simpson \cite{simpson1992}, we know that there are isomorphisms of sets of points
\[
\mc{M}_B(X, n)\cong \mc{M}_{dR}(X, n)\cong \mc{M}_H(X, n). \tag{$\star$}
\]
Moreover, the nonabelian Hodge correspondence states that much more is true: indeed, the following is a consequence of results in \cite{simpson1994-2}. 
\begin{thm}\label{nonab}Denote by $\phi: \mc{M}_B(X, n)\rightarrow\mc{M}_{dR}(X, n)$ and $\psi:\mc{M}_{dR}(X, n)\rightarrow\mc{M}_{H}(X, n)$ the bijections ($\star$). Then, one has that
\begin{itemize}
\item[(1)] $\phi$ in an isomorphism of the associated complex analytic spaces;
\item[(2)] $\psi$ is a homeomorphism of topological spaces.
\end{itemize}
\end{thm} 
The fact that Theorem \ref{nonab}.~(1) is not just a set-theoretic bijection, but an isomorphism of $complex$ analytic spaces, is a key ingredient in the proof of our main result, as it enables us to transfer the formal isomorphism given by the Isosingularity theorem to another formal isomorphism between different spaces. This will become clear in Section \ref{isosec} below.\\
In what follows we shall state the Isosingularity theorem in a form that is slightly different from the original statement, but which is equivalent and more suitable for our purposes. The reader should refer to \cite[Theorem 10.6]{simpson1994-2} for a proof of these results.
\begin{thm}[Isosingularity]\label{isosing}For any point $x\in \mm_{deR(X, n)}$ there is a canonical isomorphism between the formal completions of $\mc{M}_{dR}(X, n)$ at $x$ and $\mc{M}_{H}(X, n)$ at $\psi(x)$.
\end{thm}
\begin{rem} One can deduce an important corollary from the above theorem using a result of Artin (\cite[Corollary 2.6]{artin}): indeed, using this result one can prove that Theorem \ref{isosing} implies that $\mm_H(X, n)$ and $\mm_B(X, n)$ are locally \'etale isomorphic at corresponding points. We will make use of this consequence later in the paper when studying the singularities of $\mm_H(X, n).$
\end{rem}
As a consequence, one can relate formal completions of the spaces $\mm_B(X, n)$ and $\mm_H(X, n)$ at corresponding points. To this end, one needs the following result, which relates formal and analytic completions of a variety at a point. The reader should refer to \cite[Chapter 13]{taylor} for a proof of this result and a detailed treatment of the relations between the complex algebraic and the analytic points of view.
\begin{prop}\label{formal} Let $V$ be an algebraic variety and let $V^{an}$ denote the space $V$ considered as a complex analytic space. For $x$ a point in $V$, there is an isomorphism of locally ringed spaces
\[
\widehat{V}_x\cong \widehat{V_x^{an}}.
\]
\end{prop}

With the above proposition at hand, one can prove the following theorem, which is the crucial tool to transfer the results about symplectic resolutions from the context of character varieties to that of Higgs bundle moduli spaces.
\begin{thm}\label{formal-neig} There is an isomorphism between the formal completions of the spaces $\mc{M}_B(X, n)$ and $\mm_H(X, n)$ at corresponding points.
\end{thm}
\begin{proof}
The result is a consequence of Theorem \ref{isosing}, Proposition \ref{formal} and part (1)~of Theorem \ref{nonab}. Indeed, for $x$ a point in $\mc{M}_B(X, n)$ one has the following chain of isomorphisms
\[\widehat{\mc{M}_B(X, n)_x} \cong \widehat{\mc{M}_B(X, n)_x^{an}}\cong\widehat{\mc{M}_{dR}(X, n)_{x'}^{an}}\cong \widehat{\mc{M}_{dR}(X, n)_{x'}}\cong\widehat{\mc{M}_{H}(X, n)_{x''}},
\]
where $x'=\phi(x)$ and $x''=\psi(x')$ and the first and the third isomorphisms come from Proposition \ref{formal}, the second is a consequence of Theorem \ref{nonab}, and the fourth is precisely the Isosingularity Theorem.
\end{proof}
\section{Symplectic resolutions: from character varieties to Higgs bundles}\label{cvhbproof}
\subsection{Proof of the main result}We first recall results of Bellamy and Schedler \cite[Proposition 8.5, Corollary 8.16]{bellamy-schedler} about the nature of the singularities of $\mf{X}(g, n)$ and the existence of symplectic resolutions. Then, via Theorem \ref{formal-neig}, we prove that the analogous result holds for the variety $\mc{M}_H(X, n)$. We will use Simpson's notation $\mc{M}_B(X, n)$ for the character variety $\mf{X}(g, n)$.
\begin{thm}\label{charsymp}\cite[Proposition 8.5]{bellamy-schedler} The Poisson variety $\mc{M}_B(X, n)$ is a symplectic singularity.
\end{thm}
\begin{thm}\label{nosymp}\cite[Corollary 8.16]{bellamy-schedler} Suppose $g>1$ and $(g, n)\neq (2,2)$. Then, the symplectic singularity $\mc{M}_B(X, n)$ does not admit a symplectic resolution.
\end{thm}

Although we will not prove the above results, one should note that two different strategies are proposed in \cite{bellamy-schedler} for their proofs: 
\begin{itemize}
\item[(A)] Since a symplectic resolution is a crepant resolution, to prove that the former can not exist, it suffices to prove that the latter does not exist. To this end, it is well-known that if $Y$ is a normal variety which is factorial and has terminal singularities, then $Y$ does not admit a crepant resolution. In \cite[Theorem 8.15]{bellamy-schedler} it is shown that $\mm_B(X, n)$ has these properties under the assumptions of Theorem \ref{nosymp};
\item[(B)] one may also prove Theorem \ref{nosymp} by noting that if a symplectic resolution exists, then the same is true for the formal (or \'etale) neighbourhood at every point. This gives an alternative proof of the above result because in \cite[Remark 1.21]{bellamy-schedler} it is pointed out that the formal neighbourhood of $(0, \dots, 0)$ in the $SL(n, \mb{C})$-character variety is isomorphic to the formal neighbourhood of a certain quiver variety, which in turn, from the proof of \cite[Theorem B]{kaledin}, does not admit a symplectic resolution when $g>1$ and $(n, g)\neq (2, 2)$.
\end{itemize} 
To prove Theorem \ref{nosympfin}, we will adopt strategy (B) via the use of Theorem \ref{formal-neig}, and then apply strategy (A) \'etale locally. 
\begin{prop}\cite[Theorem 6] {namikawa}\label{nami} Let $Y$ be a complex algebraic variety. Then $Y$ is a symplectic singularity if and only if $Y$ has rational Gorenstein singularities and the regular locus $U$ of $Y$ admits an everywhere non-degenerate holomorphic closed 2-form.
\end{prop}
We now recall an important property of the moduli space $\mm_H(X, n)$, proved by Simpson in \cite{simpson1994-2}.

\begin{prop}\cite[Theorem 11.1]{simpson1994-2} \label{higgsirred}The moduli space $\mm_H(X, n)$ is irreducible and of dimension $2n^2(g-1)+2$.
\end{prop}
From the above, one can prove a key fact about the singularities of the moduli space $\mc{M}_H(X, n)$. In order to do this, we will need the following result, which gives an estimate on the codimension of the strictly semistable locus of the moduli space $\mm_H(X, n)$, \tit{i.e.} the locus given by Higgs bundles which are semistable but not stable, $\mm_H^{stps}(X, n)=\mm_H(X, n)\setminus \mm_H^s(X, n)$. By the results mentioned in Remark \ref{poly}, $\mm_H^{stps}(X, n)$ coincides with the locus of strictly polystable Higgs bundles, hence the notation $\mm_H^{stps}(X, n)$.
\begin{lem}\label{stps} The following holds true:
\[
\mr{codim}(\mm_H^{stps}(X, n))\geq 2.
\]
\end{lem}

\begin{proof}
In this proof we will use the shortened notations $\mm_n$ and $\mm_n^{stps}$ for $\mm_H(X, n)$ and $\mm_H^{stps}(X, n)$, respectively. The first step is to give an explicit description of the locus $\mm^{stps}$: from Definition \ref{higgsdefn}, a polystable Higgs bundle which is not stable is the direct sum stable Higgs bundles. Therefore, for any partition $\textbf{n}=(n_1, \dots, n_k)$ of $n$ there is a set-theoretic map 
\[
\nu_{\textbf{n}}:\mm_{\textbf{n}}:=\mm^s_{n_1}\times\dots\times\mm^s_{n_k}\rightarrow \mm_n,\ \ \  ((E_1, \phi_1), \dots, (E_k, \phi_k))\mapsto(\oplus E_i, \oplus\phi_i).
\]
Note that this map $\nu_{\textbf{n}}$ is algebraic: this follows from the fact that taking direct sums is functorial and well-defined on isomorphism classes of Higgs bundles. Up to the action of a permutation group on the image, the map $\nu_{\textbf{n}}$ is injective so that 
\[
\dim(\mm_{\textbf{n}})=\dim\mr{Im}(\nu_{\textbf{n}}).
\]
Moreover, it is clear that 
\[
\mm_n^{stps}=\bigcup_{\textbf{n}\in\mc{P}(n)}\mr{Im}(\nu_{\textbf{n}}),
\]
here $\mc{P}(n)$ denotes the set of partitions of $n$. Therefore, the following holds true:
\[
\dim \mm_n^{stps}=\mr{max}_{\textbf{n}\in \mc{P}(n)}\{\dim \mm_{\textbf{n}}\}.
\]
Let $\textbf{\underline{n}}=(n_1, \dots, n_k)\in \mc{P}(n)$ be a partition such that the maximum above is attained and let $k=l(\textbf{\underline{n}})$ be its length: note that $1\leq k\leq n$. Then, be above equality can be written as 
\[
\dim \mm_n^{stps}= \sum_{i=1}^k(2n_i^2(g-1)+2).
\]
The desired estimate can be written as 
\[
\dim \mm_n^{stps}\leq \dim\mm_n-2,
\]
which, using the calculation above, is
\[
\sum_{i=1}^k(2n_i^2(g-1)+2)=2(g-1)\left(\sum_{i=1}^kn_i^2\right)+2k\leq 2(g-1)\left(\sum_{i=1}^kn_i\right)^2,
\]
but this is clearly seen to hold true and, thus, the proof is concluded.
\end{proof}
\begin{thm}\label{hbsing}Assume that $g(X)\geq 2$. Then, the moduli space $\mm_H(X, n)$ is a symplectic singularity. 
\end{thm} 
\begin{proof}In order to prove the theorem one needs to verify that the hypotheses of Proposition \ref{nami} are satisfied. To this end, one needs to prove that $\mm_H(X, n)$ is normal, its singularities are rational Gorenstein, and the smooth locus of $\mm_H(X, n)$ admits a holomorphic symplectic form. The first two properties are proved by noting that being normal and rational Gorenstein are \'etale local properties and, thus, by the Isosingularity Theorem, $\mm_H(X, n)$ is normal and has rational Gorenstein singularities if and only if the same holds for $\mm_B(X, n)$: but this is true from Theorem \ref{charsymp}. For the third part, it is a well known result (\cite{hitchin1987, simpson1994-1}, see also \cite{wells}) that $\mm_H^s(X, n)$ admits a hyperk\"ahler structure and, thus, a holomorphic symplectic structure (this is true, more generally, for Higgs bundles of arbitrary degree $d$). Such a holomorphic symplectic structure can be extended to the smooth locus: indeed, by Lemma \ref{stps} we know that the codimension of the complement of the stable locus inside the smooth locus is at least 2, and the stable locus is dense in the smooth locus since the latter is irreducible (by Proposition \ref{higgsirred}). Therefore, all the hypotheses of Proposition \ref{nami} are satisfied and the theorem is proved.
\end{proof}

We are now ready to prove the main result of the paper.
\begin{thm}\label{nosympfin} When $g>1$ and $(g, n)\neq (2,2)$, the moduli space $\mc{M}_H(X, n)$ does not admit a symplectic resolution.
\end{thm}
\begin{proof}
Suppose by contradiction that when $g>1$ and $(g, n)\neq (2,2)$ such a resolution $\rho:\widetilde{\mm}\rightarrow \mm_H(X, n)$ exists and let $x$ be the point in $\mm_H(X, n)$ that corresponds to the trivial representation, call it $\mr{Id}$, in the Betti moduli space $\mm_B(X, n)$ via the homeomorphism given by Theorem \ref{nonab}.~(2). Then, by the Isosingularity Theorem, there is an \'etale neighbourhood $V$ of $x$ in $\mm_H(X, n)$ isomorphic to an \'etale neighbourhood $U$ of $\mr{Id}$ in $\mm_B(X, n)$. Furthermore, by assumption, via this \'etale local isomorphism, from the resolution $\rho$, one can construct a symplectic resolution $\tilde{\rho}$ of $U$. But $U$ is factorial and terminal: factoriality of $U$ follows from the proof of \cite[Theorem 8.15]{bellamy-schedler}; on the other hand, the fact that $U$ has terminal singularities follows from \cite[Corollary 1]{naminote}. Since this is a contradiction, the theorem follows.
\end{proof}

Given the above theorem and the resolution constructed in the case $(g, n)=(2, 2)$ in \cite{kiem-yoo}, the only case left to consider is that of elliptic curves. Recall that, from \cite{bellamy-schedler}, it is possible to show that, for $X$ an elliptic curve, the character variety $\mf{X}(g, n)$ does admit a symplectic resolution, (see \cite[Proposition 8.13]{bellamy-schedler}). Moreover, an analogous result can be obtained for the moduli space $\mm_H(X, n)$ using a result shown in E. Franco's PhD Thesis \cite{franco} (see also \cite{franco-gp-newstead}), which relies on techniques from \cite{hitchin1987}. 
\begin{thm}\cite[Theorem 4.19]{franco-gp-newstead} \label{elliptic} Consider the moduli space $\mm_H(X, n, d)$, where $X$ is an elliptic curve, and let $h=\mr{gcd}(n, d)$. Then, there exists an isomorphism 
\[
\alpha_{n, d}: \mr{Sym}^hT^*X\rightarrow \mm_H(X, n, d).
\]
\end{thm}
With such an explicit geometric description of $\mm_H(X, n, d)$, one can prove the following result. 
\begin{thm}\label{symp-ell}Let $X$ be an elliptic curve. Then, the moduli space $\mm_H(X, n)$ is a symplectic singularity and it admits a symplectic resolution
\[
\mr{Hilb}^n T^*X\longrightarrow \mm_H(X, n).
\]
\end{thm}
\begin{proof}
We know from the Theorem \ref{elliptic} that the moduli space $\mm_H(X, n)$ is isomorphic to the $n$-th symmetric power of the cotangent bundle to the elliptic curve $X$. Via this isomorphism, we can induce a (generic) symplectic structure on $\mm_H(X,n)$ so that, by definition, the isomorphism $\alpha_{n, 0}$ is, generically, a symplectomorphism. Moreover, it is a well-known fact (\cite[Example 2.4]{fu}) that, given a smooth symplectic surface $S$, for any $n\geq 1$, the variety $\mr{Sym}^nS$ is a symplectic singularity and there exists a symplectic resolution
\[
\mr{Hilb}^nS\rightarrow\mr{Sym}^nS.
\]
Therefore, setting $S=T^*X$ the theorem follows. 
\end{proof}
\begin{rem}
E. Franco has pointed out to us that the map $\alpha_{n, d}$ is actually a symplectomorphism over the stable locus, using the hyperk\"ahler structure on $\mm_H(X, n)$, so that the generic symplectic structure constructed in Theorem \ref{symp-ell} is the usual one.
\end{rem}
\subsection{Future directions}A natural question that one may want to address is whether there are analogous results in the context of twisted character varieties and moduli spaces of semistable Higgs bundles of arbitrary degree. The former varieties are defined as follows: taking $g, n$ and $d$ to be non-negative integers, the twisted character variety $\mf{X}(g, n, d)$ is
\[
\mf{X}(g, n, d):=\left\{ (A_1, \dots, A_g, B_1, \dots, B_g) \in GL(n, \mb{C})\ |\ \prod_{i=1}^g[A_i, B_i]=e^{\frac{2\pi i d}{n}}I\right\}.\]
In this case, much less is known: indeed, it is not even clear whether the formal non-abelian Hodge correspondence holds true. Moreover, it is not known whether the twist introduced in the previous definition affects the nature of the singularities. To prove the analogue of Theorem \ref{nosymp}, due to the lack of an Isosingularity theorem in this case, it may be better to use strategy (A) outlined above. On the other hand, it is an interesting open question if such a Isosingularity statement holds in this more general setting. A detailed study will appear in future work.

\printbibliography
\end{document}